\documentclass[reqno]{amsart}
\usepackage{amsmath}
\usepackage{amssymb, amsthm, amsfonts, tikz-cd, mathrsfs,mathtools,stmaryrd,enumitem, algorithmic,float,comment,tikz,hyperref}
\usepackage[toc]{appendix}
\usetikzlibrary{backgrounds}
\usetikzlibrary{decorations.pathreplacing}
\usepackage{fullpage}
\usepackage{xcolor}

\usepackage{tikz}
\usetikzlibrary{calc,decorations.pathreplacing}

\hypersetup{
  colorlinks   = true, 
  urlcolor     = blue, 
  linkcolor    = blue, 
  citecolor   = red 
}

\usepackage{comment}

\usepackage[capitalize,nameinlink]{cleveref}
\crefname{theorem}{Theorem}{Theorems}
\crefname{lemma}{Lemma}{Lemmas}
\crefname{claim}{Claim}{Claims}
\crefname{prop}{Proposition}{Propositions}
\crefname{figure}{Figure}{Figures}

\newtheorem{theorem}{Theorem}

\newtheorem{lemma}[theorem]{Lemma}
\newtheorem{claim}[theorem]{Claim}
\newtheorem{corollary}[theorem]{Corollary}

\newtheorem{conj}[theorem]{Conjecture}

\newtheorem{prop}[theorem]{Proposition}

\newtheorem*{claim*}{Claim}

\theoremstyle{remark}
\newtheorem*{remark*}{Remark}

\theoremstyle{definition}
\newtheorem{definition}[theorem]{Definition}

\numberwithin{theorem}{section}

\renewcommand{\phi}{\varphi}

\renewcommand{\leq}{\le}
\renewcommand{\geq}{\ge}

\newcommand{\cF}{\mathcal F}
\newcommand{\cE}{\mathcal E}

\newcommand{\cA}{\mathcal A}
\newcommand{\cB}{\mathcal B}

\newcommand{\cS}{\mathcal S}
\newcommand{\cP}{\mathcal P}

\newcommand{\cU}{\mathcal{U}}

\def\1{\mathbbm{1}}

\def\mc{\mathcal }

\def\eqdef{\overset{\rm{def}}{=}}

\renewcommand{\le}{\leqslant}
\renewcommand{\ge}{\geqslant}

\newcommand{\lpr}[1]{\left(#1\right)}

\newcommand{\abs}[1]{\left\lvert#1\right\rvert}

\newcommand{\lmid}{\,\middle|\,}

\newcommand{\VC}{\textup{VC}}

\usepackage[
backend=biber,
style=numeric,
maxnames=99,
giveninits=true
]{biblatex}
\renewbibmacro{in:}{%
  \ifentrytype{article}{}{\printtext{\bibstring{in}\intitlepunct}}}

\title{Uniform Set Systems with Uniform Witnesses}
\makeatletter
\newcommand\thankssymb[1]{\textsuperscript{\@fnsymbol{#1}}}
\makeatother

\author[Ting-Wei Chao]{Ting-Wei Chao\thankssymb{1}}
\author[Zixuan Xu]{Zixuan Xu\thankssymb{1}}
\author[Dmitrii Zakharov]{Dmitrii Zakharov\thankssymb{1}}

\thanks{\thankssymb{1}Department of Mathematics, Massachusetts Institute of Technology, Cambridge, MA, USA. Email: {\tt \{twchao, zixuanxu,zakhdm\}@mit.edu}}

\begin{document}

\begin{abstract}
Frankl and Pach \cite{FP84} and Erd\H{o}s \cite{Erd84} conjectured that any $(d+1)$-uniform set family $\mathcal{F}\subseteq \binom{[n]}{d+1}$ with VC-dimension at most $d$ has size at most $\binom{n-1}{d}$ when $n$ is sufficiently large. Ahlswede and Khachatrian \cite{AK97} showed that the conjecture is false by giving a counterexample of size $\binom{n-1}{d}+\binom{n-4}{d-2}$. 

For a set family $\mathcal{F}\subseteq \binom{[n]}{d+1}$, the condition that its VC-dimension is at most $d$ can be reformulated as follows: for any $F\in\mathcal{F}$, there exists a set $B_F\subseteq F$ such that $F\cap F'\neq B_F$ for all $F'\in\mathcal{F}$. In this direction, the first author, Xu, Yip, and Zhang \cite{CXYZ25} conjectured that the bound $\binom{n-1}{d}$ holds if we further assume that $|B_F|=s$ for every $F\in \cF$ and for some fixed $0\leq s\leq d$.

The case $s=0$ is exactly the Erd\H{o}s--Ko--Rado theorem, and the cases $s\in \{1,d\}$ were proved in the paper by the first author, Xu, Yip, and Zhang~\cite{CXYZ25}. In this short note, we show that the conjecture holds when $s\leq d/2$, and the maximal constructions are stars. Moreover, we construct non-star set families of size $\binom{n-1}{d}$ satisfying the condition for $d/2<s\leq d-1$, which suggests that the problem is substantially different in these cases.
\end{abstract}

\maketitle

\vspace{-1.5em}

\section{Introduction}
Let $\cF\subseteq \binom{[n]}{d+1}$ be a $(d+1)$-uniform set system over the ground set $[n]\eqdef \{1,\dots, n\}$. The \emph{Vapnik-Chervonenkis dimension (VC-dimension)} of $\cF$, denoted as $\VC(\cF)$ is defined to be the size of the largest set $S\subseteq [n]$ such that every subset of $S$ is realized as an intersection $S\cap F$ for some $F\in \cF$, i.e. for all $S'\subseteq S$, there exists $F\in \cF$ such that $S' = S\cap F$.
Since every set family $\cF\subseteq \binom{[n]}{d+1}$ has VC-dimension at most $d+1$, it is natural to ask for the maximum size of $\cF$ with VC-dimension at most $d$. Frankl and Pach \cite{FP84} proved that any such set family has size at most $\binom{n}{d}$. Then Frankl and Pach \cite{FP84} and Erd\H{o}s \cite{Erd84} conjectured that the upper bound can be improved to $\binom{n-1}{d}$ when $n$ is large enough. However, Ahlswede and Khachatrian constructed a counterexample of size $\binom{n-1}{d}+\binom{n-4}{d-2}$, and this was later generalized by Mubayi and Zhao \cite{MZ07} where they constructed a class of constructions attaining the same size. The upper bound was recently improved to $\binom{n-1}{d}+O_d(n^{d-2})$ by Yang and Yu \cite{YY25}, matching the second order term of the family constructed in \cite{MZ07} up to a multiplicative constant. For more works along this line, see \cite{CXYZ25, GXYZZ26,WXZ25}.

Notice that in the case when $\cF$ is a $(d+1)$-uniform family, we can equivalently formulate the condition that $\VC(\cF)\le d$ as follows: for each $F\in \cF$, there exists a subset $B_F\subseteq F$ such that for any $F'\in \cF$, we have $F\cap F'\neq B_F$. We call $B_F$ a \emph{witness} of $F$ in the sense that it cannot be realized as an intersection of $F$ with any set $F'\in\cF$. In this note, we are interested in the case where all witness sets $B_F$ have the same size. 
\begin{definition}[$s$-witness family]
    Let $n\ge d+1$ and $0\le s\le d$. A family $\cF\subseteq \binom{[n]}{d+1}$ is an \emph{$s$-witness family} if it satisfies the following condition: for every $F\in \cF$, there exists $B_F\subseteq F$ of size $s$ such that $F\cap F'\neq B_F$ for every $F'\in \cF$.
\end{definition}

A $0$-witness family $\cF$ is exactly an intersecting family since we have $F\cap F'\ne \varnothing$ for all $F,F'\in \cF$. The well-known Erd\H{o}s--Ko--Rado Theorem states the following upper bound for intersecting families.

\begin{theorem}[Erd\H{o}s--Ko--Rado~\cite{EKR61}]\label{thm:erdos-ko-rado}
    Let $d\ge 0$ and $n\ge 2(d+1)$. If $\cF\subseteq \binom{[n]}{d+1}$ is an intersecting family, then $|\cF|\le \binom{n-1}{d}$. Moreover, when $n > 2(d+1)$, equality holds if and only if $\cF$ is a star.
\end{theorem}

As a generalization of the Erd\H{o}s--Ko--Rado Theorem, the first author, Xu, Yip, and Zhang made the following conjecture in \cite{CXYZ25}.

\begin{conj}[\cite{CXYZ25}]\label{conj:uniform-witness}
    Let $n\ge 2(d+1)$ and $0\le s\le d$. Suppose $\cF\subseteq \binom{[n]}{d+1}$ is an $s$-witness family. Then $|\cF|\le \binom{n-1}{d}$.
\end{conj}

It is clear that when $s = 0$, \cref{conj:uniform-witness} follows from \cref{thm:erdos-ko-rado}. In \cite{CXYZ25}, the authors confirmed \cref{conj:uniform-witness} in the case where $s \in \{1,d\}$ and $n$ is sufficiently large. Moreover, they showed that equality holds if and only if $\cF$ is a star when $s\in \{1,d\}$ and remarked that it might be plausible to conjecture that the maximum set families achieving the bound in \cref{conj:uniform-witness} must be stars for all $0\le s\le d$.

In this note, we prove \cref{conj:uniform-witness} in the case when $s \le d/2$ and $n$ is sufficiently large.
\begin{theorem}\label{thm:d>2s}
    Let $1\le s\le d/2$ and let $n$ be sufficiently large. Suppose $\cF\subseteq \binom{[n]}{d+1}$ is an $s$-witness family. Then $|\cF|\le \binom{n-1}{d}$. Moreover, if $|\cF| = \binom{n-1}{d}$, then $\cF$ is a star.
\end{theorem}
Together with \cref{thm:erdos-ko-rado} and the results in \cite{CXYZ25}, the remaining open cases of \cref{conj:uniform-witness} are when $d/2 < s\le d-1$. In these remaining open cases, we give a construction that is not a single star achieving the size $\binom{n-1}{d}$, illustrating the barriers to adapting current methods as all the proofs of the known cases show that $\cF$ must be a single star if equality is achieved.

\begin{theorem}\label{thm:two-star}
    For any $n\ge s+3$ and $d/2< s \leq d-1$, there exists an $s$-witness family $\cF\subseteq \binom{[n]}{d+1}$ of size $\binom{n-1}{d}$ that is not a star.
\end{theorem}

We will prove \cref{thm:two-star} by giving an explicit construction that is contained in the union of two star families but not contained in a single star family. Recall that the \emph{piercing number} $\tau(\cF)$ is the cardinality of the smallest set $T$ such that $F\cap T\neq\varnothing$ for all $F\in \mc F$. In particular, the construction given in the proof of \cref{thm:two-star} has piercing number $2$. Furthermore, we can generalize the construction in \cref{thm:two-star} and get $s$-witness families of size $\binom{n-1}{d}$ with arbitrary large piercing number.

\begin{theorem}\label{thm:more-star}
    Let $d, s, k \ge 1$ be integers such that $d/2<s\leq d-k+1$ and $n \ge \max( k+(k-1)s, d+2k-1)$. There exists an $s$-witness family $\mc F \subseteq {[n] \choose d+1}$ of size ${n-1\choose d}$ such that $\tau(\mc F) = k$.
\end{theorem}

We remark that our construction gives several pairwise non-isomorphic examples (up to a permutation of the ground set) proving \cref{thm:two-star}, but the construction for \cref{thm:more-star} gives only one such example (up to a permutation of the ground set). Therefore, we will present both the proof of \cref{thm:two-star} and the proof of \cref{thm:more-star} in \cref{sec:construction} for completeness despite the fact that they share many similarities.\footnote{In an earlier arXiv version of this paper, \cref{thm:more-star} was stated as a question in the special case $k=3$. ChatGPT 5.4 Pro managed to successfully answer this question and the construction in \cref{subsec:more_star} is a generalization of its response (see \url{https://chatgpt.com/share/69b8e49c-6df0-8009-8180-52161f034bc2} for the conversation).}

Since our proof of \cref{thm:d>2s} proceeds by approximating $\cF$ by a single star, it does not generalize easily to the cases where $s > d/2$. Thus, we believe new ideas are required to prove the remaining cases of \cref{conj:uniform-witness}.

\subsection{Paper Organization} In \cref{sec:proof} we present the proof of \cref{thm:d>2s}. Then in \cref{sec:construction} we give the construction that proves \cref{thm:two-star,thm:more-star}.

\medskip\noindent \textit{Notations.} For sets $A,B\subseteq [n]$, we denote their symmetric difference as $A\triangle B = (A\setminus B)\cup (B\setminus A)$. We always consider $s$ and $d$ as fixed and $n\to \infty$. We use $O_d(\cdot)$ and $\Omega_d(\cdot)$ to suppress leading constants depending on $d$.

\medskip\noindent{\textbf{Acknowledgments.}}
We thank the organizers of the SLMath workshop ``Algebraic and Analytic Methods in Combinatorics'' (March 2025) attended by the authors where this project was initiated. We would also like to thank Zixiang Xu for comments on an early draft of this note. Finally, we thank Zach Hunter for pointing out typos in the first draft posted on arXiv.

\section{Proof of \cref{thm:d>2s}}\label{sec:proof}

Our proof consists of three steps split into the three subsequent subsections. In the first step presented in \cref{subsec:structural}, we analyze the structure of an $s$-witness set family $\cF\subseteq \binom{[n]}{d+1}$ for $1\le s\le d/2$. In particular, we show in \cref{lem:modeling} that we can capture the intersection properties of $\cF$ using set families of constant size. Then in \cref{lem:injection}, we define a certain injection of a subset of $\cF$ into $\binom{[n]}{d}$ that will be the main object that we analyze later on. Next, in \cref{subsec:star-approx}, we show that any family $\cF$ that is close to achieving the maximum size can be approximated by a star family. Finally, in \cref{subsec:stability}, we perform a stability analysis, showing that in fact any set family achieving size $\binom{n-1}{d}$ must be a single star, concluding the proof of \cref{thm:d>2s}.

From now on, we fix $1\leq s\le d/2$ and let $\cF\subseteq \binom{[n]}{d+1}$ be an $s$-witness family. In particular, for $F\in \cF$, we always fix a choice of its witness $B_F\subseteq F$ of size $s$ so that $F\cap F'\neq B_F$ for any $F'\in \cF$.

Let us recall the Erd\H{o}s-Rado sunflower lemma that will be used in the proof of \cref{lem:modeling}. For $r\ge 2$, a family of $r$ sets $H_1,\dots, H_r\subseteq [n]$ is called a {\em $k$-sunflower} if there exists $S\subseteq [n]$ such that $H_i\cap H_j = S$ for all $i\ne j\in [r]$.

\begin{lemma}[Erd\H{o}s-Rado sunflower lemma~\cite{ErdosRado1960_sunflower}]\label{lem:sunflower}
    Let $k$ and $r\ge 2$ be positive integers. If $\cF\subseteq \binom{[n]}{k}$ has size $|\cF|\ge k!(r-1)^k$, then $\cF$ contains an $r$-sunflower.
\end{lemma}

The bound in \cref{lem:sunflower} was later improved by Alweiss-Lovette-Wu-Zhang~\cite{AlweissLovettWuZhangAnnals2021} and Bell--Chueluecha--Warnke~\cite{BellChueluechaWarnke2021}, but the bound stated in \cref{lem:sunflower} suffices for our purposes. Now we begin the proof of \cref{thm:d>2s}.

\subsection{Structural Analysis}\label{subsec:structural}

In this section, we prove that one can capture the intersection properties in $\cF$ using set families of constant size. First, observe that we can interpret the uniform witness condition as follows. For any $B\in \binom{[n]}{s}$ and $F,F'\in \cF$ such that $B_F = B$ and $B\subseteq F'$, we must have~$(F\setminus B)\cap (F'\setminus B)\ne \varnothing$. In other words, the two families $\{F\setminus B\mid  F\in \cF,\,B_F=B\}$ and $\{F'\setminus B\mid F' \in \cF,\,B\subseteq F'\}$ are cross-intersecting. 

Define $\cB = \{B\in \binom{[n]}{s}\mid \exists F\in \cF \text{ such that } B_F = B\}$ to be the set of all witnesses for $\cF$. We have the following modeling lemma that captures for each $B\in \cB$ the structure of the sets $F\in \cF$ containing $B$.

\begin{lemma}[Modeling Lemma]\label{lem:modeling}
    For every $B\in \cB$, there exists a set family $\cA_B\subseteq \binom{[n]\setminus B}{\le d+1-s}$ consisting of sets of size at most $d+1-s$ satisfying the following:
    \begin{enumerate}
        \item $\cA_B$ is an intersecting family, i.e. $A\cap A'\ne \varnothing$ for all $A,A'\in \cA_B$.\label{item:intersecting}
        \item  For every $F\in \cF$ such that $B_F = B$, there exists $A\in \cA_B$ such that $A\subseteq F$.\label{item:witness}
        \item  For every $F\in \cF$ such that $B\subseteq F$ (possibly $B_F\ne B$), we have $A\cap F\ne \varnothing$ for all $A\in \cA_B$.\label{item:contain}
        \item  $|\cA_B|=O_d(1)$.\label{item:AB-size}
    \end{enumerate}
\end{lemma}

\begin{proof}
    We construct $\cA_B$ iteratively. Initialize $\cA \eqdef\{F\setminus B \mid F\in \cF, B_F=B\}$. If $\cA$ contains a $(d+2)$-sunflower $A_1,\dots, A_{d+2}\in \cA$ with a core $S\subseteq [n]\setminus B$, then we remove the sets $A_1,\dots, A_{d+2}$ from $\cA$ and add the set $S$ to $\cA$. We repeat this operation until $\cA$ is $(d+2)$-sunflower-free and set $\cA_B$ as the final ($d+2$)-sunflower-free family. We show that $\cA_B$ satisfies all the properties in \cref{lem:modeling}.

    First, note that the initial family $\cA_0 = \{F\setminus B \mid F\in \cF, B_F=B\}$ satisfies properties (1)-(3). Indeed, $\cA_0$ must be an intersecting family, because if there exists $A\cap A' = \varnothing$, then the sets $(A\cup B),  (A'\cup B) \in \mc F$ violate the property that $B$ is the witness of $A\cup B$ and $A'\cup B$. Property (2) is trivially satisfied by definition. Property (3) also follows from the fact that $B$ is the witness for sets of the form $A\cup B$ where $A\in \cA_0$. For a set $F\supseteq B$, we must have $F\cap (A\cup B)\ne B$ and thus $F\cap A\ne \varnothing$ for all $A\in \cA_0$. 

    Now we show that the operation of replacing each sunflower in $\cA$ by its core preserves properties (1)-(3). Suppose $\cA$ satisfies properties (1)-(3). We first show that $\cA$ is intersecting after replacing a $(d+2)$-sunflower $A_1,\dots, A_{d+2}\in \cA$ with the core $S\subseteq [n]\setminus B$. Let $T\in \cA\setminus \{A_1,\dots, A_{d+2}\}$ and write $A_i = S\sqcup D_i$ for all $i\in [d+2]$. We want to check that $T\cap S \neq \varnothing$. 
    Suppose for contradiction that $T\cap S = \varnothing$, then we must have $T\cap D_i\ne \varnothing$ for all $i\in [d+2]$. Since $D_1,\dots, D_{d+2}$ are pairwise disjoint, we must have that $T\cap D_i$ are pairwise disjoint for all $i\in [d+2]$. If follows $|T|\ge \sum_{i = 1}^{d+2}|T\cap D_i|\ge d+2$ contradicting the fact that $|T| \leq d+1-s$. Therefore we must have $T\cap S\ne\varnothing$ and thus the family$\cA$ after the replacement operation remains intersecting. Similarly, Property (3) holds because given $F\in \cF$ such that $B\subseteq F$, we have $F\cap A_i \neq \varnothing$ for all $i\in [d+2]$ and thus we must have $F\cap S\ne \varnothing$ by the same argument. Property (2) is clearly preserved since $S\subseteq A_i$ for all $i\in [d+2]$.

    Finally, we upper bound the size of $\cA_B$. The family $\cA_B$ is ($d+2$)-sunflower-free and consists of sets of size at most $d+1-s$. We must have $\cA_B$ is ($d+2$)-sunflower-free at every uniformity. Thus by \cref{lem:sunflower} we have
    \[|\cA_B|\le \sum_{k = 1}^{d+1-s} k!(d+1)^k=O_d(1).\]
    This completes the proof.
\end{proof}

Define 
\[
\cF_B = \{F \in \cF\mid \exists A\in \cA_B\text{ such that } A\cup B \subseteq F\}.
\]
In particular, we have $F \in \cF_B$ if $B_F=B$, but $\cF_B$ may also include other sets. Note that at this point we no longer need to keep track of the particular choice of $B_F$ for each $F\in \cF$, instead we record that 
\[\cF = \bigcup_{B \in \cB} \cF_B.\]
For $B \in \cB$, define $\alpha_B = \min_{A \in \cA_B} |A|$.

\begin{claim}\label{claim:FB-size}
    For any $B\in \cB$, we have
    \[|\cF_B| = O_d\lpr{n^{d+1-s-\alpha_B}}.\]
\end{claim}

\begin{proof}
   Since each $F\in \mc F_B$ contains an element of $\mc A_B$ as a subset, we get by double counting and the definition of $\alpha_B$:
   \[|\cF_B|\le |\cA_B|\cdot{n-s-\alpha_B \choose d+1-s-|A|} = |\cA_B|\cdot  O_d\lpr{n^{d+1-s-\alpha_B}}.\]
   By \cref{item:AB-size} in \cref{lem:modeling}, we have $|\cA_B| = O_d(1)$ and thus we have the desired bound.
\end{proof}

Define $\cB_1 = \{B \in \cB \mid \alpha_B=1\}$ and $\cB_{\ge 2} = \cB \setminus \cB_{1}$. For $B \in \cB_1$ let $x_B$ be the unique element so that $\{x_B\}\in \cA_B$ ($x_B$ is unique since $\mc A_B$ is an intersecting family). Since $\cA_B$ is intersecting, if $\{x_B\}\in \cA_B$, then we can redefine $\cA_B = \{\{x_B\}\}$ and the properties in \cref{lem:modeling} are still satisfied. Now we define an injection of a subset of $\cF$ into $\binom{[n]}{d}$ which will be a crucial part of our proof.

\begin{lemma}[Injective mapping $F\rightarrow E$]\label{lem:injection}
    For $B \in \cB_1$, define $\cE_B = \{F \setminus \{x_B\} \mid F \in \cF_B\}$ and let $\cE = \bigcup_{B\in \cB_1} \cE_{B}$. Then we have the following properties:
    \begin{enumerate}
        \item If $F \in \cF$, $B \in \cB_1$ and $E \in \cE_B$ are such that $E \subseteq F$, then we have $F = E \cup \{x_B\}$.\label{item:injective:1}
        
        \item For $B\neq B' \in \cB_1$ with $x_{B} \neq x_{B'}$, we have $\cE_B\cap \cE_{B'} = \varnothing$.\label{item:injective:2}
        
        \item For $B, B' \in \cB_1$, let
        \[\cU(B, B') = \left\{E \in {[n] \choose d} \lmid B\cup B' \subseteq E, ~ x_{B}, x_{B'} \not\in E\right\}.\]
        Then we have $\cU(B, B') \cap \cE = \varnothing$ for any $B, B' \in \cB_1$ with $x_B \neq x_{B'}$.\label{item:injective:3}
    \end{enumerate}
\end{lemma}

The idea of constructing an injection from $\mc F$ into ${[n] \choose d}$ appears in earlier works as well (see \cite[Claim 3.1]{CXYZ25} and \cite{YY25}). However, our implementation of this idea is different and, in particular, \cref{item:injective:3} in \cref{lem:injection} is new.

\begin{proof}
    We first prove \cref{item:injective:1}. From the assumption, we have $B\subseteq E\subseteq F$. By \cref{item:contain} in \cref{lem:modeling}, we know that $F\cap A\ne \varnothing$ for all $A\in \cA_B$. Since $B\in \cB_1$, we have $\cA_B = \{\{x_B\}\}$. Thus, $x_B\in F$ and we must have $F = E\cup \{x_B\}$.

    To prove \cref{item:injective:2}, suppose for contradiction that we have $E\in \cE_B\cap \cE_{B'}$ for $B$ and $B'$ with $x_B\ne x_{B'}$. Then we have $E = F\setminus \{x_B\} = F'\setminus \{x_{B'}\}$ for some $F,F'\in\cF$. By \cref{item:injective:1} we must have $F = E\cup \{x_B\} = E\cup \{x_{B'}\}$ and thus $x_B =x_{B'}$, contradicting the assumption.

    Finally, to prove \cref{item:injective:3}, consider $B, B'\in \cB_1$ with $x_B\ne x_{B'}$. Assume for contradiction that there exists $E \in \cU(B, B')\cap \cE$, then we must have $E = F'' \setminus \{x_{B''}\}$ for some $B'' \subseteq F'' \in \cF$. On the other hand, since $B\subseteq F''$, by \cref{item:injective:1} we must have $x_{B''} = x_B$. Similarly, since $B'\subseteq F''$, we have $x_{B'} = x_B$. Therefore we have $x_B = x_{B''}$ contradicting the assumption that $x_B\ne x_{B'}$.
\end{proof}

Define $\cF_1 = \bigcup_{B\in \cB_1}\cF_B$ and $\cF_{\ge 2} = \bigcup_{B\in \cB_2}\cF_B$. Note that by definition $|\cF|\le |\cF_1| + |\cF_{\ge 2}|$. Then \cref{lem:injection} implies that the map $\phi: \cF_1\to \binom{[n]}{d}$ defined by $F \mapsto F\setminus \{x_B\}$ is an injection. In particular, $\cE$ is the image $\phi(\cF)\subseteq \binom{[n]}{d}$, so we have $|\cF_1|= |\cE|$, and it suffices to upper bound $\cE$. From \cref{item:injective:3} in \cref{lem:injection}, the set of $\cU(B,B')$ for $B,B'\in \cB_1$ with $x_B\ne x_{B'}$ is contained in the complement of $\cE$.

In order to have $\cU(B,B')\neq\varnothing$, we need $x_B\notin B'$ and $x_{B'}\notin B$. Thus, we define
\[\cP\eqdef\{(B,B')\in \cB_1\times\cB_1\mid x_B\neq x_{B'},x_B\notin B',x_{B'}\notin B\}\quad\text{and}\quad\cU\eqdef\bigcup_{(B,B')\in \cP}\cU(B,B').\]
Then we have the following useful inequality for bounding $|\cF_1|$ as a direct corollary of \cref{lem:injection}.

\begin{corollary}\label{cor:injection-ineq}
    We have
    \begin{equation}\label{eq:injection-ineq}
    |\cF_1| + |\cU| \le {n \choose d}.    
    \end{equation}
\end{corollary}

\subsection{Star approximation}\label{subsec:star-approx}
In this section, we prove that any $s$-witness family with $1\le s\le d/2$ and size close to maximal can be approximated by a star family up to a bounded symmetric difference.

Recall that $\cF_1 = \bigcup_{B\in \cB_1}\cF_B$, $\cF_{\ge 2} = \bigcup_{B\in \cB_{\geq 2}}\cF_B$, and $|\cF|\le |\cF_1| + |\cF_{\ge 2}|$. We have the following estimate.

\begin{claim}\label{claim:F2-size}
    We have $|\cF_{\ge 2}| = O_d(n^{d-1})$.
\end{claim}

\begin{proof}
    Note that we have
    \[|\cF_{\ge 2}|\le \sum_{B\in \cB_{\ge 2}}|\cF_B|.\]
    For each $B\in \cB_{\ge 2}$, we have $\alpha_B\ge 2$ and thus by \cref{claim:FB-size} we have $|\cF_B| = O_d(n^{d-s-1})$. Therefore, we obtain
    \[|\cF_{\ge 2}|\le |\cB_{\ge 2}|\cdot O_d(n^{d-s-1}) \le \binom{n}{s}\cdot O_d(n^{d-s-1}) = O_d(n^{d-1}).\qedhere\]
\end{proof}

Now we state our result on star approximation.

\begin{prop}\label{prop:star-approx}
    Suppose $d \ge 2s$ and $n\ge 2(d+1)$. For any $K\ge 0$, if $\cF\subseteq \binom{[n]}{d+1}$ is an $s$-witness family with size 
    \[|\cF|\ge {n-1 \choose d} - K n^{d-1},\]
    then there exists a star family $\cF^*$ with $|\cF\triangle \cF^*| \le O_d((K+1) n^{d-1})$.
\end{prop}

\begin{proof}
    Note that by adjusting the implied constant in $O_d$, we may assume that $K \le c n$ for an arbitrarily small fixed constant $c>0$ depending on $d$. 

    First note that by \cref{claim:F2-size}, we have $|\cF_{\ge 2}| = O_d(n^{d-1})$, so we may focus on $\cF_1$. We first prove a lower bound on $|\cB_1|$ in the following claim.

    \begin{claim}\label{claim:B1-size-lb}
        We have $|\cB_1|=\Omega_d(n^s)$.
    \end{claim}

    \begin{proof}
    From \cref{claim:FB-size}, we have
    \[|\cF_1|\le \sum_{B\in \cB_1}|\cF_B|\le |\cB_1|\cdot O_d(n^{d-s}).\]
    By rearranging, we get $|\cB_1|=\Omega_d\left(|\cF_1|n^{s-d}\right)$. Recall that we have $|\cF_1|\ge |\cF| - |\cF_{\ge 2}|$, $|\cF_2|= O_d(n^{d-1})$, and $|\cF| \ge \binom{n-1}{d} - Kn^{d-1}=\Omega_d(n^d)$, where we use the assumption $K \le cn$. It follows that $|\cF_1| = \Omega_d(n^d)$, and hence $|\cB_1|=\Omega_d(n^s)$.
    \end{proof}
    
    For $x\in[n]$, let $\cB_{1}(x) \eqdef \{B\in \cB_1\mid x_B = x\}$. Define $x_0 \eqdef \arg\max_{x\in [n]}|\cB_{1}(x)|$ and let $M\ge 0$ be such that $|\cB_{1}(x_0)| = |\cB_1| - M$. Define the star family
    \[
    \cF^* = \left\{F^*\in \binom{[n]}{d+1}\lmid x_0\in F^*\right\}.
    \]
    Note that we have
    \[|\cF\triangle \cF^*| = |\cF\setminus \cF^*| + |\cF^*\setminus \cF|.\]
    By assumption, we have 
    \[|\cF^*\setminus \cF| = |\cF^*| - |\cF^*\cap \cF|\le (|\cF^*|-|\cF|) - |\cF \setminus \cF^*|\le Kn^{d-1} + |\cF\setminus \cF^*|.\]
    Thus it suffices to bound $|\cF\setminus \cF^*|$. By definition and \cref{claim:FB-size}, we have 
    \[|\cF\setminus \cF^*|\le |\cF_{\ge 2}| + \sum_{B\in \cB_1\setminus \cB_{1}(x_0)}|\cF_B|\le O_d(n^{d-1}) + M\cdot O(n^{d-s}).\]
    So now our goal is to upper bound $M$. We first bound it in terms of $\abs{\cU}$.
\begin{claim}\label{lem:UM}
    We have 
    \[\abs{\cU}=\Omega_d\left(n^{d-2s}\abs{\cB_1}\left(M-2\binom{n-1}{s-1}\right)\right).\]
\end{claim}
\begin{proof}
    For any $(B, B') \in \cP$, since $s\le d/2$, we have $|\cU(B, B')| \ge \binom{n-2s-2}{d-2s}$. Note that every $E \in {[n] \choose d}$ can contain $B\cup B'$ for at most $\binom{d}{s}\cdot \binom{d-s}{s}$ pairs $(B, B') \in \cP$. Thus, we have
    \[
        \left|\cU\right| \ge \binom{d}{s}^{-1} \binom{d-s}{s}^{-1}\sum_{(B, B')\in \cP} |\cU(B, B')|
        \ge \binom{d}{s}^{-1} \binom{d-s}{s}^{-1} \binom{n-2s-2}{d-2s}|\cP|
        = \Omega_d(n^{d-2s}\abs{\cP}).
    \]

    Now, we lower bound $\abs{\cP}$. 
    By the inclusion-exclusion principle, we know that $|\cP|$ is at least
    \begin{align*}
        |\left\{(B, B')\in \cB_1\times \cB_1 \mid x_{B} \neq x_{B'}\right\}| - |\left\{(B, B')\in \cB_1\times \cB_1 \mid x_{B} \in B'\right\}| - |\left\{(B, B')\in \cB_1\times \cB_1 \mid x_{B'} \in B\right\}|.
    \end{align*}
    Note that $|\left\{(B, B')\in \cB_1\times \cB_1 \mid x_{B} \neq x_{B'}\right\}|\geq |\cB_1|M$ since there are at least $M$ choices for $B'$ once $B$ is chosen (see the definition of $M$). On the other hand, we know that $|\left\{(B, B')\in \cB_1\times \cB_1 \mid x_{B} \in B'\right\}|\leq\nobreak |\cB_1|\binom{n-1}{s-1}$ since there are at most $\binom{n-1}{s-1}$ choices for $B'$. Similarly, we have $|\left\{(B, B')\in \cB_1\times \cB_1 \mid x_{B'} \in B\right\}|\leq |\cB_1|\binom{n-1}{s-1}$. Therefore, we can conclude that
    \[|\cP|\ge |\cB_1| M - 2 |\cB_1| {n-1\choose s-1}.\]
    This completes the proof.
\end{proof}

    Now, we are ready to upper bound $M$.
    \begin{claim}\label{claim:M-ub}
        We have $M= O_d((K+1)n^{s-1})$.
    \end{claim}

    \begin{proof}
    If $M\leq 2 {n-1\choose s-1}$, then the claim is trivially true. Otherwise, combining \cref{claim:B1-size-lb,lem:UM}, we get
    \[\left|\cU\right| =\Omega_d\left(n^{d-s}\left(M-2 {n-1\choose s-1}\right)\right)=\Omega_d(n^{d-s}M)-O_d(n^{d-1}).\] 
    Now by \cref{cor:injection-ineq}, we have
    \[\left|\cU\right| \leq \binom{n}{d}-|\cF_1|\leq \binom{n}{d}-|\cF|+|\cF_{\geq 2}|=\binom{n-1}{d}-|\cF|+O_d(n^{d-1}). \]
    Combining with the assumption that $|\cF| \ge {n-1 \choose d} - K n^{d-1}$, we can conclude that 
    \[K n^{d-1}=\Omega_d(n^{d-s}M)-O_d(n^{d-1}).\]
    After rearranging, we get $M = O_d((K+1)n^{s-1})$.
    \end{proof}

    Thus by \cref{claim:M-ub}, we can conclude that 
    \[|\cF\triangle \cF^*|\le Kn^{d-1} + 2|\cF\setminus \cF^*|\le Kn^{d-1} + O_d(n^{d-1}) + O_d((K+1)n^{d-1}) = O_d((K+1)n^{d-1}),\]
    and \cref{prop:star-approx} follows.
\end{proof}

\subsection{Stability analysis}\label{subsec:stability}
We are now ready to perform the stability analysis and finish the proof of \cref{thm:d>2s}.

\begin{proof}[Proof of \cref{thm:d>2s}]
Let $\cF\subseteq \binom{[n]}{d+1}$ be an $s$-witness family with $1\le s\le d/2$ and suppose $|\cF|\ge \binom{n-1}{d}$. Then by \cref{prop:star-approx}, there exists $x_0\in [n]$ such that the star family centered at $x_0$ defined as 
\[\cF^* = \left\{F^*\in \binom{[n]}{d+1}\lmid x_0\in F^*\right\}\]
satisfying $|\cF\triangle \cF^*| = O_d(n^{d-1})$. Let $L = |\cF^*\setminus \cF|$ and note that $L= O_d(n^{d-1})$. Our goal is to show that $L=0$.

We will give estimates regarding the following sets:
\begin{enumerate}
    \item $\cS_1 \eqdef \{B\in \cB \mid x_0\in B\}$;
    \item $\cS_2 \eqdef \{B\in \cB_{\ge 2}\mid x_0\not\in B\}$;
    \item $\cB_1(x_0) \eqdef \{B\in \cB_1\mid x_B=x_0\}$;
    \item $\cB_1'\eqdef \{B\in \cB_1\mid x_B\neq x_0,\, x_0\notin B\}$.
\end{enumerate}

\begin{claim}\label{lem:1}
    We have
    \[\sum_{B \in \cS_1} |\cF_B| = O_d(L/n).\]
\end{claim}

\begin{proof}
    Fix $B\in \cS_1$. We first lower bound the number of sets $F\in \cF^*\setminus \cF$ with $B\subseteq \cF$.
    For any $A\in\cA_B$, we know that the number of sets $F\in \cF^*$ that contain $A\cup B$ is at most $\binom{n-s-\abs{A}}{d+1-s-\abs{A}}=O_d(n^{d-s})$. On the other hand, there are $\binom{n-s}{d+1-s}$ sets in $\cF^*$ that contain $B$. Since any set $F\in \cF^*\cap \cF$ with $B\subseteq \cF$ must contain $A\cup B$ for some $A\in\cA_B$, we know that the number of $F\in \cF^*\setminus \cF$ with $B\subseteq \cF$ is at least $\binom{n-s}{d-s+1}-O_d(n^{d-s})\abs{\cA_B}=\Omega_d(n^{d-s+1})$.
    
    Since every $F^* \in \cF\setminus \cF$ contains at most ${d\choose s-1}$ sets $B \in \cS_1$, we obtain
    \[L = |\cF^* \setminus \cF| \ge \binom{d}{s-1}^{-1}\cdot \Omega_d(n^{d-s+1})\cdot |\cS_1| .\]
    Thus, we have $|\cS_1|=O_d(Ln^{s-d-1})$. On the other hand, by \cref{claim:FB-size}, for any $B\in\cS_1$, we know that $|\cF_B| = O_d(n^{d+1-s-\alpha_B}) = O_d(n^{d-s})$, so
    \[\sum_{B \in \cS_1} |\cF_B| = O_d(n^{d-s}) \cdot|\cS_1|=  O_d(L/n).\qedhere\]
\end{proof}

\begin{claim}\label{lem:2}
    We have
    \[\sum_{B \in \cS_2} |\cF_B| = O_d(L/n).\]
\end{claim}
\begin{proof}
    The proof is similar to the one of \cref{lem:1}. Fix $B \in \cS_2$. We first lower bound the number of sets $F\in \cF^*\setminus \cF$ with $B\subseteq \cF$.
    For any $A\in\cA_B$, we know that $\abs{A}\geq 2$ since $B\in\cB_{\geq 2}$. By a direct calculation, the number of sets $F\in \cF^*$ that contains $A\cup B$ is at most $\binom{n-s-1-\abs{A}}{d+1-s-\abs{A}}=O_d(n^{d-s-1})$ (equality is achieved when $x_0\in A$ and $\abs{A}=2$). On the other hand, there are $\binom{n-s-1}{d-s}$ sets in $\cF^*$ that contains $B$. Since any set $F\in \cF^*\cap \cF$ with $B\subseteq \cF$ must contain $A\cup B$ for some $A\in\cA_B$, we know that the number of $F\in \cF^*\setminus \cF$ with $B\subseteq \cF$ is at least $\binom{n-s-1}{d-s}-O_d(n^{d-s-1})\abs{\cA_B}=\Omega_d(n^{d-s})$.

    Since every $F^* \in \cF\setminus \cF$ contains at most ${d+1\choose s}$ sets $B \in \cS_2$, we obtain
    \[L = |\cF^* \setminus \cF| \ge \binom{d+1}{s}^{-1}\cdot \Omega_d(n^{d-s})\cdot |\cS_2| .\]
    Thus, we have $|\cS_2|=O_d(Ln^{s-d})$. On the other hand, by \cref{claim:FB-size}, for any $B\in \cS_2$, we know that $|\cF_B| = O_d(n^{d+1-s-\alpha_B}) = O_d(n^{d-s-1})$, so
    \[\sum_{B \in \cS_2} |\cF_B| = O_d(n^{d-s-1}) \cdot|\cS_2|=  O_d(L/n).\qedhere\]
\end{proof}

Since we are focusing on the sets $F\in \cF$ that are not captured by $\cF^*$, we need restrict our injection to the sets in $[n]\setminus\{x_0\}$. Thus, we define
\begin{alignat*}{2}
    &\cF_1(x_0)\eqdef \bigcup_{B\in\cB_1(x_0)}\cF_B, \quad\quad &\cF_1'\eqdef& \bigcup_{B\in\cB_1'}\cF_B,\\
    &\cP'\eqdef\cP\cap (\cB_1(x_0)\times \cB_1'),&\quad 
    \cU'\eqdef&\bigcup_{(B,B')\in \cP'}\cU(B,B').
\end{alignat*}
Note that $\cU'\subseteq \binom{[n]\setminus\{x_0\}}{d}$ since any set $E\in\cU(B,B')$ does not contain $x_0=x_B$ if $B\in \cB_1(x_0)$. Thus, the injection $\varphi$ we defined right after \cref{lem:injection} gives the following corollary.
\begin{corollary}\label{cor:inj2}
    We have
    \[\abs{\cF_1(x_0)}+\abs{\cF_1'\setminus \cF^*}+\abs{\cU'}\leq \binom{n-1}{d}.\]
\end{corollary}
\begin{proof}
    Note that $\varphi(\cF_1(x_0))$, $\varphi(\cF_1'\setminus \cF^*)$, and $\cU'$ are set families in $\binom{[n]\setminus\{x_0\}}{d}$ and they are pairwise disjoint by \cref{lem:injection}.
\end{proof}

We also need the following variant of \cref{lem:UM}.
\begin{lemma}\label{lem:UB1star}
    We have $\abs{\cU'}=\Omega_d(n^{d-s}\abs{\cB_1'})$.
\end{lemma}
\begin{proof}
    Similar to \cref{lem:UM}, we know that $\abs{\cU'}=\Omega_d(n^{d-2s}\abs{\cP'})$. Thus, we need to lower bound $\abs{\cP'}$.
    
    Note that for each fixed $B'\in \cB_1'$, there are at most $\binom{n-2}{s-1}$ possible $B\in \cB_1(x_0)$ such that $x_{B'}\in B$. Thus, there are at least $\abs{\cB_1(x_0)}-\binom{n-2}{s-1}$ pairs of the form $(B,B')\in \cP'$ for each $B'\in \cB_1'$. Recall that we defined $M\ge 0$ such that $|\cB_{1}(x_0)| = |\cB_1| - M$. From \cref{claim:B1-size-lb,claim:M-ub} with $K=0$, we have
    \[\abs{\cB_1(x_0)}-\binom{n-2}{s-1}=\abs{\cB_1}-M-\binom{n-2}{s-1}=\Omega_d(n^s).\]
    Therefore, $\abs{\cP'}\geq \Omega_d(n^s\abs{\cB_1'})$, and the lemma follows.
\end{proof}

Now we are ready to finish the proof of \cref{thm:d>2s}. Note that $\cS_1\cup\cS_2\cup\cB_1(x_0)\cup\cB_1'=\cB$, and hence
\[\abs{\cF}\leq \sum_{B \in \cS_1} |\cF_B|+\sum_{B \in \cS_2} |\cF_B|+\abs{\cF_1(x_0)}+\abs{\cF_1'}.\]
From \cref{lem:1,lem:2}, we know that the first two terms in the right hand side are $O_d(L/n)$. Furthermore, we can write $\abs{\cF_1'}=\abs{\cF_1'\setminus \cF^*}+\abs{\cF_1'\cap \cF^*}$. Together with \cref{cor:inj2}, we get
\begin{align}
    \abs{\cF} &\leq  O_d(L/n)+ \abs{\cF_1(x_0)} + \abs{\cF_1' \setminus \cF^*} + \abs{\cF_1'\cap \cF^*}  \nonumber \\ 
    &\leq O_d(L/n)+\binom{n-1}{d}-\abs{\cU'}+\abs{\cF_1'\cap \cF^*}.\label{eq:finalF}
\end{align}
For any fixed $B\in \cB_1'$, we know that $|\cF_B\cap \cF^*|\leq \binom{n-s-2}{d-s-1}$ since any set $F\in \cF_B\cap \cF^*$ must contain $B\cup\{x_0,x_B\}$. Therefore, we have \begin{align}
    \abs{\cF_1'\cap \cF^*}=O_d(n^{d-s-1}\abs{\cB_1'}).\label{eq:F1'capF*}
\end{align}
Combining \cref{lem:UB1star,eq:finalF,eq:F1'capF*}, we have
\[\abs{\cF}\leq \binom{n-1}{d}+O_d(L/n)-\Omega_d(n^{d-s}\abs{\cB_1'}).\]
Recall that we assumed $\abs{\cF}\geq \binom{n-1}{d}$. It follows that $\abs{\cB_1'}= O_d(n^{s-d-1}L)$

Now, we bound $\abs{\cF\setminus \cF^*}$.
Note that if $F\in\cF\setminus \cF^*$, then $F\in\cF_B$ for some $B\in \cS_2\cup \cB_1'$. Therefore, 
\[\abs{\cF\setminus \cF^*}\leq \sum_{B \in \cS_2} |\cF_B| +\abs{\cF_1'\setminus \cF^*}\leq O_d(L/n)+O_d(n^{d-s}\abs{\cB_1'})=O_d(L/n),\]
where the second inequality follows from \cref{lem:2,claim:FB-size}. 

However, we have
\[\abs{\cF\setminus \cF^*}=\abs{\cF}-\abs{\cF^*}+\abs{\cF^*\setminus\cF}\geq \binom{n-1}{d}-\binom{n-1}{d}+L=L.\]
Therefore, $O_d(L/n)\geq L$, and hence $L=\abs{\cF^*\setminus\cF}=0$ when $n$ is sufficiently large. This implies that $\cF^*\subseteq \cF$, and it is clear that $\cF^*$ is a maximal $s$-witness family. This completes the proof.
\end{proof}

\section{Non-star constructions}\label{sec:construction}
We first prove \cref{thm:two-star} in \cref{subsec:two_star}, and then prove \cref{thm:more-star} in \cref{subsec:more_star}
\subsection{A non-star construction for $ d/2<s\leq d-1$}\label{subsec:two_star}
Now we give a construction that proves \cref{thm:two-star}.
\begin{proof}[Proof of \cref{thm:two-star}]
    Fix any integer $m\in [s+2,n-1]$. Consider the set family $\cF\eqdef\cA_1\cup\cA_2\cup\cA_{12}$ defined as follows:
    \begin{align*}
        \cA_1\eqdef&\left\{\{1\}\cup A\lmid A\in\binom{[3,n]}{d},\,\abs{A\cap [3,m]}\geq s\right\},\\
    \cA_2\eqdef&\left\{\{2\}\cup A\lmid A\in\binom{[3,n]}{d},\,\abs{A\cap [3,m]}< s\right\},\\
    \cA_{12}\eqdef&\left\{\{1,2\}\cup A\lmid A\in\binom{[3,n]}{d-1}\right\},
    \end{align*}
    It is clear that $\cF$ is not a star. Moreover, we have
    \[|\cF| = \binom{n-2}{d} + \binom{n-2}{d-1} = \binom{n-1}{d}.\]

   We choose the witness $B_F$ for each $F\in \cF$ as follows.
    \begin{itemize}
        \item $F\in \cA_{12}$: $F = \{1,2\}\cup A$ for some $A\in \binom{[3,n]}{d-1}$. We set $B_F\subseteq A$ to be any arbitrary subset of size $s$.

        \item $F\in \cA_1$: $F = \{1\}\cup A$ for some $A\in\binom{[3,n]}{d}$ such that $|A\cap [3,m]|\ge s$. We set $B_F\subseteq A\cap [3,m]$ to be any arbitrary subset of size $s$.

        \item $F\in \cA_2$: $F = \{2\}\cup A$ for some $A\in\binom{[3,n]}{d}$ such that $|A\cap [3,m]|< s$. We set $B_F\subseteq A$ to be any subset of size $s$ such that $|B_F\cap [m+1,n]|\ge d-s+1$. This is well-defined by the assumption that $|A\cap [m+1,n]|\ge d - |A\cap [3,m]|\ge d-s+1$ and $s\ge d-s+1$.
    \end{itemize}
    Now we justify that the above constructions are indeed valid witnesses.
    \begin{itemize}
        \item $F\in \cA_{12}$: Since $\{1,2\}\subseteq F$ and $\{1,2\}\cap F'\ne \varnothing$ for all $F'\in\cF$, so we have $\{1,2\}\cap (F\cap F')\ne\varnothing$. Since $B_F\cap \{1,2\} = \varnothing$, we have $B_F$ is a valid witness for $F$.

        \item $F\in \cA_1$: Since $1\in F$ and $1\not\in B_F$, any $F'\in \cF$ satisfying $F\cap F' = B_F$ must be in $\cA_2$. However, since $|F'\cap [3,m]| < s$ and $B_F\subseteq A\cap [3,m]$, we cannot have $F'\cap F = B_F$.

        \item $F\in \cA_2$: Similarly, since $2\in F$ and $2\not\in B_F$, any $F'\in \cF$ satisfying $F\cap F' = B_F$ must be in $\cA_1$. Then we must have $\abs{F'\cap [m+1,n]}\leq d-s$. Since we have $B_F\supseteq A\cap [m+1,n]$ and $|A\cap [m+1,n]\ge d-s+1$, there must exist an element in $B_F$ that is not in $F'$, so $F\cap F'\neq B_F$.
    \end{itemize}
    This shows that $\cF$ is an $s$-witness family.
\end{proof}

\subsection{A non-star construction with arbitrarily large piercing number}\label{subsec:more_star}
The construction for \cref{thm:more-star} is similar to the one given in \cref{subsec:two_star}. In particular, when the piercing number $k=2$, the following construction recovers the one in \cref{subsec:two_star} with $m=s+2$.
\begin{proof}[Proof of \cref{thm:more-star}]
    Let $U = [n] \setminus [k] = \{k+1, \ldots, n\}$. Let $S_1, S_2, \ldots, S_{k-1} \subseteq U$ be a collection of pairwise disjoint subsets of size $s$. This can be done provided that $n \ge k + (k-1)s$.
    For $i \in [k-1]$ and $m \ge s$, let us define the families
    \[
    \mc A_i(m) \eqdef \left\{ A \in {U\choose m}\,\middle|\,S_i \subseteq A  \right\},
    \]
    and 
    \[
    \mc A_k (m) \eqdef \left\{ A \in {U\choose m}\,\middle|\, S_i \not\subseteq A ~\forall i =1,\ldots, k -1  \right\}.
    \]
   Using the assumption that $2s> d$, we know that these families are pairwise disjoint and they cover the sets in ${U \choose m}$ if $m\leq d$. 
    For a non-empty subset $I \subseteq [k]$, we define
    \[
    \mc A_I \eqdef \bigsqcup_{i \in I} \mc A_i\left(d+1- |I|\right) \subseteq {U \choose d+1-|I|}.
    \]
    Let $\mc F_I \eqdef \{I\cup A \mid A\in \mc A_I\}$ and define $\mc F \eqdef \bigsqcup_{\varnothing\neq I\subseteq [k]} \mc F_I $. We claim that $\mc F$ is an $s$-witness family of size ${n-1 \choose d}$ with piercing number $k$. First, let us compute the size of $\mc F$. Note that for each $t=1, \ldots, k$, we have
    \[
    \sum_{I \in {[k]\choose t}} |\mc A_I| = \sum_{i=1}^k {k-1 \choose t-1} |\mc A_i| = {k-1\choose t-1} {n-k\choose d+1-t},
    \]
    and so
    \[
    |\mc F| = \sum_{\varnothing\ne I\subseteq [k]}|\mc A_I|  = \sum_{t=1}^k {k-1\choose t-1} {n-k\choose d+1-t} = {n-1\choose d}.
    \]
    Next, we show that $\cF$ is indeed an $s$-witness family. We begin by defining the witness sets $B_F$ for each $F\in\cF$. Suppose that $F\in\cF_I$ with nonempty $I \subseteq [k]$, so that there exists $A \in \mc A_I$ such that $F = I\cup A$. Then we define $B_F$ in the following way:
    \begin{itemize}
        \item If $S_i \subseteq A$ for some $i \in [k-1]$, then we set $B_F = S_i$.
        \item Otherwise, let $J$ be the set of indices $j \in [k-1] \setminus I$ such that $A\cap S_j \neq\varnothing$. For every $j \in J$, fix an arbitrary element $x_j \in S_j \cap A$. Let $B_F$ be an arbitrary subset of $A \setminus \{x_j\mid j \in J\}$ of size $s$. 
    \end{itemize}
    Note that in the second case we have $|J| \le k-|I|$. Thus, we have 
    \[
    |A \setminus \{x_j\mid j \in J\}| = |A|-|J| = d+1-|I| - |J|\ge d+1-k \ge s,
    \]
    and hence it is possible to pick an $s$-subset $B_F\subseteq A \setminus \{x_j\mid j \in J\}$. Furthermore, we note that $B_F\subseteq U$ for all $F\in \cF$.

    Now, we check the $s$-witness condition. That is, we check that $F\cap F'\neq B_F$ for all $F,F'\in\cF$. Suppose that $F = I\cup A$ and $F' = I'\cup A'$ with nonempty $I, I' \subseteq [k]$, $A \in \mc A_I$, and $A' \in \mc A_{I'}$. We check the condition in the following case analysis.    
    \begin{itemize}
        \item If $I\cap I' \neq \varnothing$, then $F\cap F' \supseteq I\cap I'$. Since $I\cap I'$ is disjoint from from $B_F \subseteq U$, we have $F\cap F'\neq B_F$ in this case.

        \item If $I\cap I' = \varnothing$ and $S_i \subseteq A$ for some $i \in [k-1]$, then $i \in I$ and $B_F = S_i$. It follows that $i \not\in I'$ and hence $B_F=S_i \not\subseteq A'$.

        \item Otherwise, we have $I\cap I' = \varnothing$ and $S_i \not\subseteq A$ for all $i\in [k-1]$. It follows that $k \in I$. Since $A'\in \mc F_{I'}$ and $I' \subseteq [k-1]$, there exists $i' \in I'$ such that $S_{i'} \subseteq A'$. 
        \begin{itemize}
            \item If $A \cap S_{i'} = \varnothing$, then we have
        \[
        |F\cap F'| = |A\cap A'| = |A \cap (A'\setminus S_{i'})| \le |A'|-|S_{i'}| \le d-s < s, 
        \]
        and hence $F\cap F' \neq B_F$.
        \item Otherwise, we have $A\cap S_{i'} \neq\varnothing$. By construction there is an element $x_{i'} \in A\cap S_{i'}$ which does not lie in $B_F$. On the other hand, since $S_{i'} \subseteq A'$, we clearly have $x_{i'} \in A\cap A' \subseteq F\cap F'$. Therefore, $F\cap F' \neq B_F$. 
        \end{itemize}
    \end{itemize}
    By the case analysis above, we know that $\cF$ is an $s$-witness family.

    It remains to show that $\tau(\mc F)=k$. It is clear that $[k]$ pierces every set in $\mc F$, and hence $\tau(\mc F) \le k$. Now, assume for the sake of contradiction that there is some $(k-1)$-set $T = I \sqcup R$, where $I \subseteq [k]$ and $R \subseteq U$, that intersects every set in $\mc F$. For any element $j \in [k]\setminus I$, we know that $R$ intersects every set in $\mc A_j(d)$. For each $j\in [k-1]\setminus I$, then we must have $R\cap S_j \neq\varnothing$. Otherwise, we can find a set $A\subseteq U\setminus R$ of size $d$ such that $S_j\subseteq A$ and $A\cap R=\varnothing$. This can be done since $\abs{U\setminus R}=\abs{U}-\abs{R}\geq n-k-(k-1)\geq d$. Therefore, we know that $\abs{(I\cup R)\cap S_j}=1$ for all $j\in [k-1]$, and $k\notin I$. Thus, $R$ must intersect every set in $\mc A_k(d)$. However, we can find a set $A\subseteq U\setminus R$ of size $d$ that does not contain any $S_i$ with $i\in [k-1]$ as a subset. This is because we only need to avoid the set $R$ and also avoid one element from $S_i$ for each $i\in I$. This is doable since $\abs{U}-\abs{R}-\abs{I}=n-k-(k-1)\geq d$. 
    
    We can conclude that $\cF$ is an $s$-witness family of size $\binom{n-1}{d}$ and it has piercing number $k$. This completes the proof.
\end{proof}

\printbibliography

\end{document}